\documentclass[titlepage]{article}
\usepackage{url}
\usepackage{mathrsfs}
\usepackage{amsmath,amsthm,amssymb}
\usepackage{cleveref}
\usepackage[dvipdfmx]{graphicx}
\usepackage[all]{xy}
\usepackage{color}
\usepackage{comment}
\usepackage{algorithm}
\usepackage[noend]{algpseudocode}
\usepackage{colortbl}
\numberwithin{equation}{section}

\newtheorem{thm}{Theorem}[section]
\newtheorem{prop}[thm]{Proposition}
\newtheorem{lemm}[thm]{Lemma}
\newtheorem{cor}[thm]{Corollary}
\newtheorem{defi}[thm]{Definition}
\newtheorem{eg}[thm]{Example}
\newtheorem{question}[thm]{Question}
\newtheorem{rmk}[thm]{Remark}

\newcommand{\scr}[1]{\mathscr{#1}}
\newcommand{\bb}[1]{\mathbb{#1}}

\newcommand{\Sp}{\mathscr{S}_+}
\newcommand{\Spc}{\mathscr{S}_+^{\circ}}
\newcommand{\Qnn}{\mathbb{Q}_{\geq 0}}
\newcommand{\Rnn}{\mathbb{R}_{\geq 0}}
\newcommand{\Iff}{\Longleftrightarrow}

\DeclareMathOperator{\Sign}{Sign}
\DeclareMathOperator{\supp}{supp}
\DeclareMathOperator{\SC}{SC}
\DeclareMathOperator{\sz}{sz}
\newcommand{\plus}{\makebox[12pt]{$+$}}
\newcommand{\zero}{\makebox[12pt]{$0$}}
\newcommand*{\minus}{\makebox[12pt]{$-$}}
\newcommand{\floor}[1]{\lfloor#1\rfloor}
\newcommand{\ceil}[1]{\lceil#1\rceil}

\newcommand*{\poly}{\mathit{poly}}
\newcommand*{\expt}{\mathit{exp}}

\begin{document}

\title{On the Complexity of Interpolation by Polynomials with Non-negative Real Coefficients}

\author{Katsuyuki BANDO\footnote{kbando@ms.u-tokyo.ac.jp}, 
Eitetsu KEN\footnote{yeongcheol-kwon@g.ecc.u-tokyo.ac.jp}, \&
Hirotaka ONUKI\footnote{onuki-hirotaka010@g.ecc.u-tokyo.ac.jp}}








\maketitle
\begin{abstract}
In this paper, we consider interpolation by \textit{completely monotonous} polynomials (CMPs for short), that is, polynomials with non-negative real coefficients.
In particular, given a finite set $S\subset \bb{R}_{>0} \times \Rnn$, we consider \textit{the minimal polynomial} of $S$, introduced by Berg [1985], which is `minimal,' in the sense that it is eventually majorized by all the other CMPs interpolating $S$.
We give an upper bound of the degree of the minimal polynomial of $S$ when it exists. 
Furthermore, we give another algorithm for computing the minimal polynomial of given $S$, which utilizes an order structure on sign sequences. 
Applying the upper bound above, we also analyze the computational complexity of algorithms for computing minimal polynomials including ours.
%
%
\end{abstract}

\section{Introduction}
By celebrated Descartes' rule of signs, we have that each univariate polynomial with non-negative real coefficients (sometimes called \textit{completely monotonous} polynomial (CMP)) is uniquely characterized among others as an interpolation polynomial of an appropriate finite set of points.
\begin{eg}
The polynomial $f(X)=X^{100}$ is a unique CMP interpolating $S=\{(1,1),(2,2^{100}),(3,3^{100})\}$.
Indeed, if $g\in \Rnn[X]\setminus \{X^{100}\}$ interpolates $S$, the polynomial $g-X^{100}$ has three positive roots $1,2,3$, which contradicts Descartes' rule of signs since the number of sign changes of $g-X^{100}$ is at most two.
\end{eg}
This is a radically different situation compared with the case of polynomials with unrestricted real coefficients.
A further interesting point is that the minimum number of points we need to identify a CMP $f$ is related to the number of the terms of $f$ rather than the degree of $f$.
(See Corollary \ref{meaningofmu} for the details.)

A line of research on the reverse direction, that is, given a finite set $S=\{(a_{i},b_{i})\}_{i=1}^{n} \subset \bb{R}_{>0} \times \Rnn$, finding the unique CMP $f$ interpolating $S$ if such a polynomial exists was essentially initiated by \cite{firstminimalpolynomial} and further developed in \cite{minimalpolynomials} and related works.
After those, \cite{Interpolationbypolynomialswithnonnegativecoefficients} independently considered the problem in the context of feedback control systems.

\cite{firstminimalpolynomial} introduced the notion of \textit{minimal polynomials} (of given $S$), and it is essential when we consider the interpolation problem above; $f$ is a unique CMP interpolating $S$ if and only if it is a minimal polynomial of some $T \subsetneq S$. (cf. Corollary \ref{uniqueCMPisminimalpolynomial}). \cite{firstminimalpolynomial} gave a concrete algorithm to compute minimal polynomials in a special case,
\cite{Thielcke} gave two heuristics: one using linear programming and the other without it, and \cite{minimalpolynomials} established the theoretical background of \cite{Thielcke}.
Furthermore, \cite{Interpolationbypolynomialswithnonnegativecoefficients} took a more geometrical approach and gave another algorithm deciding the existence of interpolating CMP's by exploring the surface hyperplanes of a \textit{clam}, a particular convex set depending on $S$.

However, the computational complexity of these algorithms, which is related to degree upper bounds for $f$ above, has not been analyzed.
It is natural to refer to results on sparse interpolation since $f$ should be $\# S$-sparse if $f$ is the minimal polynomial of $S$.
Indeed, when we consider the case when $f$ is the unique CMP interpolating $S$, then $f$ is $(\# S -1)$-sparse, and the degree upper bound given in Theorem 4.5 in \cite{degreeupperboundforsparseinterpolation} can be applied.
Nevertheless, it is still unclear whether this degree upper bound can be applied to the minimal polynomials which are precisely $\# S$-sparse. 
It is beyond the scope of the theory of general sparse interpolation since it is trivial to find an $\# S$-sparse polynomial in $\mathbb{R}[X]$ interpolating $S$.

In this paper, we focus on analyzing minimal polynomials and give a degree upper bound for them, described in terms of the size of a reasonable representation of the input $S$ (Theorem \ref{discreteupperbound}). 
As an application, we evaluate the computational complexity of simple algorithms computing minimal polynomials by linear programming (Theorem \ref{thm:complexity-of-algo-by-LP}).
Furthermore, we give a new linear-programming-free algorithm for finding minimal polynomials which works along a particular order structure on sign sequences (Algorithm \ref{findtheminimalpolynomial}).  
We evaluate its complexity, too (Theorem \ref{complexityoffindtheminimalpolynomial}).
This algorithm is somehow similar to the one in \cite{Interpolationbypolynomialswithnonnegativecoefficients}, but ours are more refined so that the relation to the order structure is clearer.
We believe that this algorithm is theoretically meaningful rather than practically since it reveals the structure of minimal polynomials by developing the theory of sign sequences.

The article is organized as follows:

In section \ref{Preliminaries}, we set up our notation, review some of the known results on Descartes' rule of signs, and locate the work of \cite{minimalpolynomials} on minimal polynomials in our context.

In section \ref{A degree upper bound for minimal polynomials}, we prove our main result, a degree upper bound for minimal polynomials, and give a sketch of the evaluation of the complexity of simple algorithms using linear programming to find minimal polynomials.

In section \ref{An order structure of sign sequences}, we introduce an order structure on sign sequences, which is helpful in describing manipulations of CMPs relying on Descartes' rule of signs and its optimality.

In section \ref{An algorithm}, we present another linear-programming-free algorithm for finding minimal polynomials, which searches the object based on the order structure introduced in section \ref{An order structure of sign sequences}.
We also mention the relation to the algorithm in \cite{Interpolationbypolynomialswithnonnegativecoefficients}.

In section \ref{Open Problems}, we cast some open problems which clarify the precise location of our work in the course of research.

\section{Preliminaries}\label{Preliminaries}

First, we set our notation:

\begin{itemize}
\item $\bb{R}$ denotes the set of real numbers, and $\bb{Q}$ denotes the set of rational numbers. 
When we consider algorithms, since we want the inputs and outputs to have finite expressions, we restrict the 
coefficients of polynomials to be rational.
However, the theory behind those algorithms applies to the case of real coefficients, so we use $\bb{R}$ for generality, too.

\item $\omega$ denotes the set of natural numbers (including $0)$.
\item For $n, m \in \omega$, the interval $\{x \in \omega \mid n \leq x \leq m\}$ (which might be empty) is denoted by $[n,m]$. $[1,m]$ is also denoted by $[m]$.

\item Let $\bb{R}_{>0}:=\{x\in \bb{R}\mid x>0\}$ and $\bb{R}_{\geq 0}:=\{x\in \bb{R}\mid x\geq 0\}$.
We also define $\bb{Q}_{>0}$ and $\bb{Q}_{\geq 0}$ similarly.
\item For a set $A$, its cardinality is denoted by $\# A$. 
\item For a sequence $s=(s_{i})_{i \in \omega} \in \{-,0,+\}^{\omega}$ of signs, \textit{the support of $s$}, denoted by $\supp(s)$, is defined by;
\[
\supp(s):=\{i\in \omega\mid s_i\neq 0\}\subseteq \omega.
\]

\item The set $\scr{S}$ of finite sequences of signs is defined by;
\[
\scr{S}:=\left\{s \in \{-,0,+\}^{\omega} \mathrel{}\middle|\mathrel{} \# \supp(s) < \infty \right\}.
\]

For $s \in \scr{S}$, the degree $\deg(s)$ of $s$ is defined by;
\[
\deg(s):=\max\supp(s) \in \omega \cup \{-\infty\},
\]
where we set $\max \emptyset = -\infty$.

We write an element $s=(s_i)_{i\in \omega}\in \scr{S}$ by
\[
s_0\cdots s_d \ \mbox{or} \ (s_{0}\cdots s_{d})
\]
if $\deg(s) \leq d$.
\item 
Define 
\[
\Sign\colon \bb{R}[X]\to \scr{S} ;\; \sum_{i=0}^\infty a_iX^i\mapsto (\Sign(a_i))_{i\in \omega},
\]
where $\Sign(a_i)\in \{-,0,+\}$ is the sign of $a_i\in \bb{R}$.

Note that $\deg(\Sign(f(X))) = \deg(f(X))$ (we adopt a convention $\deg(0) = -\infty$).
\item Put $\scr{S}_+:=\scr{S}\cap\{0,+\}^{\omega}$.
Note that the restriction of $\supp$ to $\Sp$: 
\[\supp\colon \Sp \to \{S\subset \omega\mid \#S<\infty\}\]
is a bijection.
\item $\Rnn[X]$ denotes the set of polynomials with non-negative real coefficients (\textit{completely monotonous polynomials}, or \textit{CMP}s in short), that is, 
\[
\Rnn[X]:=\left\{f\in \bb{R}[X] \mathrel{}\middle|\mathrel{} \Sign(f)\in \scr{S}_+ \right\}.
\]
Note that each $f(X) \in \Rnn[X] \setminus \Rnn$ gives a strictly increasing function on $\bb{R}_{>0}$.
We also define $\Qnn[X]$ similarly.

\item $\SC(t)$ is the number of sign changes in $t$, defined by
\[
\SC(t):=\#\left\{(i,j)\in \omega^2 \mathrel{}\middle|\mathrel{}
 \begin{array}{l}
i<j\\
\{t_i,t_j\}=\{-,+\} \\
t_k=0\text{ for any $k\in \left[i+1,j-1\right]$}
\end{array}\right\}.
\]

\end{itemize}

The following well-known Descartes' rule of signs is crucial in this work:
\begin{thm}
Let $f \in \bb{R}[X] \setminus\{0\}$ and $p$ be the number of positive roots of $f$ counting their multiplicities. 
Then $(\SC(\Sign(f))-p)$ is an even non-negative integer. 
In particular, 
 \[p \leq \SC(\Sign(f)).\]
\end{thm} 

Furthermore, it is also known that the bound above is optimal in the following sense:

\begin{thm}[\cite{AlbouyFu}]
Let $M$ be a finite multiset of positive reals and $\sigma\in \scr{S}$.
If $(\SC(\sigma) - \#M)$ is an even non-negative integer, then there exists $f \in \bb{R}[X]$ such that:
\begin{itemize}
 \item The set of all positive roots (formalized as a multiset reflecting their multiplicities) of $f$ is precisely $M$,
 \item $\Sign(f) = \sigma$.
\end{itemize}
\end{thm}


\begin{cor}\label{realization} 
Let $S$ be a finite set of positive reals and $\sigma\in \scr{S}$.
If $\#S \leq \SC(\sigma)$, then there exists $f \in \bb{R}[X]$ such that:
\begin{itemize}
 \item The set of all positive roots of $f$ is precisely $S$.
 \item $\Sign(f) = \sigma$.
\end{itemize}
\end{cor}

\begin{rmk}
As for the optimality, there had been a series of studies on it \cite{DescartesRuleofSignoptimalityfirst, DescartesRuleofSignsoptimalcoolconstruction} before \cite{AlbouyFu}, gradually relaxing the conditions on sign sequences $\sigma$. See section 1 of \cite{AlbouyFu} for a clear exposition of the history.
\end{rmk}

Our starting point is the following definition and corollary (Corollary \ref{meaningofmu}) of Descartes' rule of signs and its optimality, which essentially already appeared in \cite{minimalpolynomials} as a notion of \textit{pair structures} (cf. Proposition \ref{equivalence}):

\begin{defi}
For $s=(s_i)_{i\in\omega}\in \Sp$, the number $\mathfrak{d}(s)$ denotes the maximum of the number of sign changes in $t\in \scr{S}$ such that $t$ can be obtained by changing some entries of $s$ to $-$ (or we can also allow $0$).
More precisely,
\begin{align*}
\mathfrak{d}(s):=&\max\left\{\, \SC(t) \mathrel{}\middle|\mathrel{} 
\begin{array}{l}t=(t_i)_{i\in\omega}\in \scr{S}\\
t_i\in \{-,s_i\}\ \text{for any $i\in \omega$}
\end{array}\right\}\\
=&\max\left\{\, \SC(t) \mathrel{}\middle|\mathrel{} 
\begin{array}{l}t=(t_i)_{i\in\omega}\in \scr{S}\\
t_i\in \{-,0,s_i\}\ \text{for any $i\in \omega$}
\end{array}\right\}\text.
\end{align*}
Note that $t_i$ can be $-$ even if $i$ is greater than> $\deg s$.

For $f\in \Rnn[X]$, put $\mathfrak{d}(f):=\mathfrak{d}(\Sign(f))$.
\end{defi}

The following observation is useful: 

\begin{lemm}\label{achieve}
Let $s\in \Sp$ and $\supp(s) \neq \emptyset$.
Let $t\in \scr{S}$ be a sign sequence obtained by changing some entries of $s$ to $0$ or $-$.
If $\SC(t)=\mathfrak{d}(s)$, then $\deg(t) \geq \deg(s) \geq 0$ and $t_{\deg(t)} = -$.
\end{lemm}

\begin{proof}
First, note that $\mathfrak{d}(s) \geq 1$ since $\supp(s) \neq \emptyset$.

Suppose $\deg(t)<\deg(s)$. Then we can increase the number of sign changes in $t$ by changing $t_{\deg(s)}$ to $+$ and $t_{\deg(s)+1}$ to $-$.
Thus $\deg(t) \geq \deg(s) \geq 0$.

Now, $t_{\deg(t)}$ is $-$ or $+$ by defnition. 
Suppose $t_{\deg(t)}=+$.
Then we can increase the number of sign changes in $t$ by changing $t_{\deg(t)+1}$ to $-$. 
Thus $t_{\deg(t)}=-$.
\end{proof}

The quantity $\mathfrak{d}(s)$ can be feasibly calculated as follows:
\begin{prop}\label{rmk:mucalculation}
Let $s=(s_i)_{i\in\omega}\in \Sp$.
Let $X_0,\ldots,X_k$ be the unique decomposition of $\supp(s)$ satisfying the following:
\begin{itemize}
\item $\supp(s)=X_0\amalg\cdots\amalg X_k$.
\item For each $i \in [0,k]$, $X_i$ is an interval in $\omega$.
\item If $0 \in \supp(s)$, $0 \in X_{0}$. Otherwise, $X_{0}=\emptyset$.
\item For each $i \in [k]$, $X_{i} \neq \emptyset$.
\item $\min X_{i}-\max X_{i-1}\geq 2$ for any $i\in [k]$.
Note that we set $\max\emptyset := -\infty$.
\end{itemize}
Let $l_{i}:=\#X_{i}$ for each $i \in [0,k]$.
Then 
\[
\mathfrak{d}(s)=l_0+2\left\lceil \frac{l_1}{2} \right\rceil+\cdots +2\left\lceil \frac{l_k}{2} \right\rceil.
\]
\end{prop}

\begin{eg}
If
$s=\underbrace{\plus\plus\plus}_{l_0=3}\zero\zero \underbrace{\plus}_{l_1=1} \zero \underbrace{\plus\plus}_{l_2=2}$,

then
$\mathfrak{d}(s)=3+2+2=7$. 

If $s=\underbrace{}_{l_0=0}\zero\underbrace{\plus\plus\plus}_{l_1=3}\zero\zero \underbrace{\plus}_{l_2=1}$,
then $\mathfrak{d}(s)=0+4+2=6$.
\end{eg}

\begin{proof}[Proof of Proposition \ref{rmk:mucalculation}]
Note that, given $n \geq 0$, $2\left\lceil \frac{n}{2} \right\rceil$ is the least even integer bounding $n$ above.

We prove the equality by induction on $k$ of the decomposition.

When $k=0$, $s=(s_{i})_{i \in \omega}$ is of the following form:
\[
s=\underbrace{\plus\cdots\plus}_{l_0}. 
\]
By changing the components of $s$ to $-$ alternately from $s_{l_{0}}$,
we obtain
\[t=\underbrace{\cdots \minus\plus\minus\plus}_{l_0}\minus\]
attaining $\SC(t) = l_{0}$.
Hence, $\mathfrak{d}(f) \geq l_{0}$.
On the other hand, if $t^{\prime}$ is obtained from $s$ by changing some components of $s$ to $-$, there is no sign change in $(t^{\prime}_{i})_{i=l_{0}}^{\infty}$ since $(s_{i})_{i=l_{0}}^{\infty}=(0,0,\cdots)$.
Therefore, $\SC(t^{\prime}) \leq l_{0}$.
Thus $\mathfrak{d}(s) = l_{0}$.

Consider the case for $k+1$.
The sign sequence $s$ is of the following form:
\[s = \underbrace{s^{\prime}}_{n} \underbrace{\zero\cdots\zero}_{m\geq 1} \underbrace{\plus\cdots\plus}_{l_{k+1}},\]
where $\supp(s^{\prime})$ is divided into intervals $X_{0}, \ldots, X_{k}$ in the manner described in the statement. 
By induction hypothesis, there exists $t$ obtained by changing some components of $(s^{\prime}00\cdots)$ to $-$ such that 
\[\SC(t) = \mathfrak{d}(s^{\prime})= l_0+2\left\lceil \frac{l_1}{2} \right\rceil+\cdots +2\left\lceil \frac{l_k}{2} \right\rceil.
\]
By Lemma \ref{achieve}, $\deg(t) \geq \deg(s) \geq 0$ and $t_{\deg(t)} = -$, so, without loss of generality, we may assume that $t$ is of the form:
\[ t= \underbrace{t^{\prime}}_{n} \underbrace{\minus\cdots\minus}_{m}.\]
(Changing $t_{n+i}$ to $-$ does not change $\SC(t)$, and neither cutting off $t_{n+m+i}$ does, given $t_{n+m-1}=-$.)

Now, changing components of $(s_{i})_{i=0}^{n+m-1}$ to $t$ and those of\\
$(s_{i})_{i
=n+m}^{n+m+2\left\lceil  \frac{l_{k+1}}{2}\right\rceil-1}$ alternately starting from $t_{n+m+2}$, we obtain the following $u$:
\[u=\underbrace{t^{\prime}}_{n} \underbrace{{-}\,\cdots\,{-}}_{m}\, \underbrace{{+}\,{-}\,{+}\,{-}\,\cdots\, {+}\,{-}}_{2\left\lceil  \frac{l_{k+1}}{2}\right\rceil} \]
Then we have 
\[\mathfrak{d}(s) \geq \SC(u) = \SC(t) + 2\left\lceil  \frac{l_{k+1}}{2}\right\rceil =  l_0+2\left\lceil \frac{l_1}{2} \right\rceil+\cdots +2\left\lceil \frac{l_{k+1}}{2} \right\rceil.\]
On the other hand, if $v$ attains $\SC(v) = \mathfrak{d}(s)$, without loss of generality, we may assume that it has the following form:
\[v=\underbrace{v^{\prime}}_{n} \underbrace{{-}\,\cdots\,{-}}_{m} \,\underbrace{v^{\prime\prime}\,{-}}_{l_{k+1}+1}.\]
(Changing $0$-entries of $s$ to $-$ does not decrease $\SC(v)$, and neither cutting off $v_{n+m+l_{k+1}+1}$ and further, given $v_{n+m+l_{k+1}}=-$.)

The number of sign changes in $(v_{i})_{i=0}^{n+m-1}$ is bounded by $\mathfrak{d}(s^{\prime})$.
Furthermore, the number of sign changes in $(v_{i})_{i=n+m-1}^{n+m+l_{k+1}}$ is bounded by $2\left\lceil \frac{l_{k+1}}{2} \right\rceil$ since it starts and ends with $-$.
Thus 
\[\mathfrak{d}(s) \leq \mathfrak{d}(s^{\prime}) + 2\left\lceil  \frac{l_{k+1}}{2}\right\rceil =  l_0+2\left\lceil \frac{l_1}{2} \right\rceil+\cdots +2\left\lceil \frac{l_{k+1}}{2} \right\rceil,\]
and this completes the proof.
\end{proof}

\begin{cor}[The meaning of $\mathfrak{d}(f)$; a corollary of Descartes' rule of signs]\label{meaningofmu}
 Let $f(X) \in \Rnn[X]$ interpolate $S = \{(a_{i},b_{i})\}_{i=1}^{n} \subset \bb{R}_{>0}\times \Rnn$ $(0 < a_{1} < \cdots< a_{n})$. 
 \begin{enumerate}
  \item\label{uniqueness} If $\mathfrak{d}(f) + 1 \leq n$, then $f$ is the unique CMP interpolating $S$.
  \item\label{manyothers} If $n \leq \mathfrak{d}(f)$, then there are infinitely many CMPs interpolating $S$.
 \end{enumerate}
\end{cor}

\begin{proof}
 
 We first consider the case \ref{uniqueness}.
 Assume $\mathfrak{d}(f) + 1 \leq n$.
 Suppose $g \in \Rnn[X]\setminus\{f\}$ also interpolates $S$.
 Then each $a_{i}$ is a positive root of the polynomial $f-g \in \bb{R}[X]\setminus \{0\}$.
 Therefore, we have 
 \[n=\#S \leq \SC(\Sign(f-g)) \leq \mathfrak{d}(f)\]
 by Descartes' rule of signs, the definition of $\mathfrak{d}(f)$, and that each coefficient of $f$ and $g$ is non-negative.
 This contradicts $\mathfrak{d}(f)+1 \leq n$.
 
 Next, we consider the case \ref{manyothers}.
 By the definition of $\mathfrak{d}(f)$, there exists $\sigma \in \mathscr{S}$ such that:
 \begin{itemize}
  \item $\sigma$ can be obtained by changing some components of $\Sign(f)$ to $-$, and
  \item $n=\#S \leq \SC(\sigma)$.
 \end{itemize}
 
 For such $\sigma$, there exists $h(X) \in \bb{R}[X]$ such that: 
 \begin{itemize}
  \item $h(a_{i}) = 0$ for each $i \in [n]$, and
  \item $\Sign(h) = \sigma$
 \end{itemize} 
 by the optimality of Descartes' rule of signs.
 
 Then, for any sufficiently small $\epsilon > 0$, $f-\epsilon h \in \Rnn[X]$ also interpolates $S$.
\end{proof}

Now, we can define the notion of a \textit{minimal polynomial}:
\begin{defi}
 Let $S=\{(a_{i},b_{i})\}_{i=1}^{n} \subset \bb{R}_{>0}\times \Rnn$ $(0<a_{1}< \cdots < a_{n})$ and $f(X)  \in \Rnn[X]$ interpolate $S$.
 We say \textit{$f$ is the minimal polynomial of $S$} if $\mathfrak{d}(f)\leq n$.
\end{defi}

\begin{eg}
$0 \in \Rnn[X]$ is the minimal polynomial of $S=\emptyset$.
\end{eg}

\begin{eg}
For each $a>0$ and $b \in \Rnn$,
$b \in \Rnn[X]$ is the minimal polynomial of $S=\{(a,b)\}$.
Especially if $b=0$, then it is the only CMP interpolating $S$, by Corollary \ref{meaningofmu}.
\end{eg}

A further example is postponed to Example \ref{baseexample}. 
We treat the uniqueness and the existence of the minimal polynomial in Theorem \ref{equivalence} and Theorem \ref{existenceofminimalpolynomial}.

We use the same term ``the minimal polynomial" as \cite{minimalpolynomials} based on the following theorem:
\begin{thm}[essentially by \cite{minimalpolynomials}]\label{equivalence}
Let $f(X) \in \Rnn[X]$ interpolate finite $S=\{(a_{i},b_{i})\}_{i=1}^{n}$ $(0<a_{1} < \cdots < a_{n})$.
Then the following are equivalent:
\begin{enumerate}
 \item\label{defofminpol} $f$ is the minimal polynomial of $S$.
 \item \label{quitestrongminimality} Any $g \in \Rnn[X]\setminus\{f\}$ interpolating $S$ can be represented as
 \[g(X) = f(X) + r(X) \prod^{n}_{i=0}(X-a_{i})\]
 with a residual polynomial $r(X)$ such that $r(t) > 0$ for any $t > 0$ (cf. \cite{firstminimalpolynomial} and Theorem 5 in \cite{minimalpolynomials}).
 \item\label{strongminimality} For any $g \in \Rnn[X]\setminus\{f\}$ interpolating $S$ and any $t > a_{n}$, it holds that $g(t) > f(t)$.
 \item\label{weakminimality} For any $g \in \Rnn[X]$ interpolating $S$ and any $t > a_{n}$, it holds that $g(t) \geq f(t)$.
  \item\label{quiteweakminimality} For any $g \in \Rnn[X]$ interpolating $S$, the leading coefficient of $g-f$ is nonnegative, that is, for sufficiently large $t$, it holds that $g(t) \geq f(t)$.
 \item\label{pairstructure} $f$ has the following pair structure (cf. \cite{minimalpolynomials}):
\[f(X) = \sum_{i=1}^{g_{n}} (\alpha_{i}+\beta_{i}X)X^{m_{i}},\]
 where 
 $\alpha_{i},\beta_{i} \in \Rnn$,
 $g_{n} := \left\lfloor\frac{n+1}{2}\right\rfloor$,
 $m_{i}+2 \leq m_{i+1}$ for $i \in [g_{n}-1]$, and
 \[\begin{cases}
  \text{$m_{1}=-1$ \& $\alpha_{1}=0$} \ &(\mbox{if $n$ is odd})\\
  m_{1}\geq 0 \ &(\mbox{if $n$ is even}).
 \end{cases}
 \]
\end{enumerate}
In particular, a minimal polynomial of $S$ is unique.
\end{thm}
\begin{rmk}
Note that \cite{minimalpolynomials} considers polynomials vanishing at zero while we are treating the whole $\Rnn[X]$ here; that is, they adopt $X\Rnn[X]$ instead of $\Rnn[X]$ in the definition of minimal polynomials.
Note that the two conventions can be converted to each other simply; let $f(X) \in \Rnn[X]$ interpolate $S=\{(a_{i},b_{i})\}_{i=1}^{n} \subset \bb{R}_{>0} \times \Rnn$ $(0<a_{1}< \cdots < a_{n})$.
 Then $f$ is the minimal polynomial of $S$ in our sense if and only if $Xf(X)$ is the minimal polynomial of $\{(a_{i},a_{i}b_{i})\}_{i=1}^{n}$ in the sense of \cite{minimalpolynomials}.
The condition (\ref{pairstructure}) is also consistent with the criterion given in \cite{minimalpolynomials} applied to $Xf(X)$.
 We stick to our convention since we think it looks more natural when we consider $\mathfrak{d}(f)$.
\end{rmk}

\begin{rmk}
It is proven that $r(X)$ in (\ref{quitestrongminimality}) is actually a CMP. 
See Theorem 5 in \cite{minimalpolynomials}.
\end{rmk}

The essence is already argued in \cite{minimalpolynomials}, but we include our version of the proof for completeness:

\begin{proof}
 (\ref{quitestrongminimality}) $\Longrightarrow$ (\ref{strongminimality}) $\Longrightarrow$ (\ref{weakminimality})  $\Longrightarrow$ (\ref{quiteweakminimality})
 immediately follows from the definitions.
 
 Also, (\ref{defofminpol}) $\Longleftrightarrow$ (\ref{pairstructure}) follows easily from the computation of $\mathfrak{d}(f)$ given in Proposition \ref{rmk:mucalculation}.
 
 (\ref{quiteweakminimality}) $\Longrightarrow$ (\ref{defofminpol}).
 We consider the contrapositive. 
 Suppose $\mathfrak{d}(f) \geq n+1$.
 Then $f \neq 0$, and there exists $\sigma \in \mathscr{S}$ such that:
 \begin{itemize}
  \item $\sigma$ is obtained by changing some components of $\Sign(f)$ to $-$.
  \item $\SC(\sigma) \geq n+1$.
 \end{itemize}
 
 Let $\sigma=\sigma_{0}\cdots \sigma_{k}$ and $\Sign(f)=s_{0} \cdots s_{l}$, where \\
 $k\geq l=\deg(f)$.
 Changing the top components $\sigma_{l} \cdots \sigma_{k}$ back to $s_{l}\cdots s_{k}= +0\cdots 0$,
 we obtain $\tau \in \mathscr{S}$ such that:
 \begin{itemize}
  \item $\tau$ is obtained by changing some components of $\Sign(f)$ to $-$.
  \item $\SC(\tau) \geq n$.
  \item $\deg(\tau)=l$ and $\tau_{l}=+$.
 \end{itemize}
 
 Now, by the optimality of Descartes' rule of signs, there exists $h \in \bb{R}[X]$ such that:
 \begin{itemize}
  \item Each $a_{i}$ $(i \in [n])$ is a root of $h$.
  \item $\Sign(h) = \tau$.
 \end{itemize}
 
 Then we have $g_{\epsilon}:=f-\epsilon h \in \Rnn[X]$ for sufficiently small $\epsilon > 0$, and $g_{\epsilon}$ interpolates $S$.
 Furthermore, the leading coefficient of $f-g_{\epsilon}$ is positive, and therefore $f(t) > g_{\epsilon}(t)$ for sufficiently large $t$.
 Thus (\ref{quiteweakminimality}) does not hold.
 
  (\ref{defofminpol}) $\Longrightarrow$ (\ref{quitestrongminimality}).
 Suppose $g \in \Rnn[X]\setminus \{f\}$ also interpolates $S$.
 Then each $a_{i}$ is a root of $f-g \in \bb{R}[X] \setminus \{0\}$, and therefore we can write
 \[g(X) = f(X) + r(X) \prod_{i=1}^{n}(X-a_{i}).\]
 The remaining problem is to show $r(t) > 0$ for any $t > 0$.

 First, by Descartes' rule of signs, we have
  \[n \leq \SC(\Sign(f-g)) \leq \mathfrak{d}(f).\]
  
  By the assumption (\ref{defofminpol}), we obtain $n=\mathfrak{d}(f)$. 
  Therefore, we see that there is no root of $(f-g)$ other than $a_{i}$'s, and, by Descartes' rule of signs again, each $a_{i}$ is a simple root; that is, $r$ has no root in $\mathbb{R}_{>0}$.
  Furthermore, by Lemma \ref{achieve}, the leading coefficient of $(f-g)$ is negative.
  Thus, the leading coefficient of $r$ is positive.
  This implies $r(t)> 0$ for any $t>0$.
\end{proof}

By Corollary \ref{meaningofmu},
we can see that  our original problem, that is, to find the unique CMP
interpolating given $S$ if it exists, is essentially the same as finding the minimal polynomial of $S$ if it exists:

\begin{cor}\label{uniqueCMPisminimalpolynomial}
Let $f(X) \in \Rnn[X]$ interpolate finite $S=\{(a_{i},b_{i})\}_{i=1}^{n}$ $(0<a_{1}<\cdots < a_{n})$.
Then the following are equivalent:
\begin{enumerate}
\item $f(X)$ is the unique CMP interpolating $S$.
\item $f(X)$ is the minimal polynomial of some $T \subsetneq S$.
\end{enumerate}
\end{cor}

\begin{eg}\label{baseexample}
Let $S=\{(a_{1},b_{1}), (a_{2},b_{2})\}$ $(0<a_{1}<a_{2}$, $0 \leq b_{1} \leq b_{2})$.
If $b_{1}=b_{2}=0$, then $0 \in \Rnn[X]$ is the unique CMP interpolating $S$ since it interpolates $S$ and $\mathfrak{d}(0)=0<\#S$. 
If $b_1=0$ and $b_2 > 0$, then there is no CMP interpolating $S$. 

Assume $b_{1} > 0$.

If $b_{1}=b_{2}$, then $b_{1} \in \Rnn[X]$ is the unique CMP interpolating $S$ since it interpolates $S$ and $\mathfrak{d}(b_{1}) = 1 < \#S$.

If $0 < b_{1} < b_{2}$, let $n := \lceil \frac{\log(b_{2}/b_{1})}{\log(a_{2}/a_{1})} \rceil$, that is, $n$ is the unique natural number satisfying
\[\left(\frac{a_{2}}{a_{1}} \right)^{n-1} < \frac{b_{2}}{b_{1}} \leq \left(\frac{a_{2}}{a_{1}} \right)^{n}.\]
Then
\[f(X) := \frac{b_{1}\left\{ \left( \frac{b_{2}}{b_{1}} - \left(\frac{a_{2}}{a_{1}}\right)^{n-1}\right)X^{n} + a_{1} \left( \left(\frac{a_{2}}{a_{1}}\right)^{n} - \frac{b_{2}}{b_{1}} \right) X^{n-1} \right\}}{(a_{2}-a_{1})a_{2}^{n-1}} \]
is the minimal polynomial of $S$.
Indeed, it is straightforward to see $f \in \Rnn[X]$, $f$ interpolates $S$, and $\mathfrak{d}(f) = 2 = \#S$.
\end{eg}

There are also criteria when minimal polynomials exist:
\begin{thm}[\cite{Thielcke}]\label{existenceofminimalpolynomial}
A set $S=\{(a_{i},b_{i})\}_{i=1}^{n} \ (0<a_{1} < \cdots < a_{n})$ has a minimal polynomial if and only if there exists $f \in \Rnn[X]$ which interpolates $S$.
\end{thm}

The following criterion enables us to take an inductive approach:

\begin{thm}[\cite{minimalpolynomials}]\label{abovecriterion}
Let $n \geq 1$,
\[S_{n+1}=\{(a_{i},b_{i})\}_{i=1}^{n+1} \ (0<a_{1} < \cdots < a_{n+1}),\] 
and let $f(X) \in \Rnn[X]$ interpolate $S_{n}:=\{(a_{i},b_{i})\}_{i=1}^{n}$.
Suppose $f(X)$ is the minimal polynomial of $S_{n}$ and $\mathfrak{d}(f)=n$, that is, $f$ is not the minimal polynomial of $S_{n-1}:=\{(a_{i},b_{i})\}_{i=1}^{n-1}$.
Then the following are equivalent:
\begin{enumerate}
 \item There exists a minimal polynomial of $S_{n+1}$.
 \item $f(a_{n+1}) \leq b_{n+1}$.
\end{enumerate}

\end{thm}

We do not include proofs of the above two theorems here.
However, the argument in the proof of Theorem \ref{inductiveupperbound} originates from a careful analysis of the proof given in \cite{minimalpolynomials}.
Furthermore, Algorithm \ref{minimalpolynomialbylinearprogramming} is based on a proof of Theorem \ref{existenceofminimalpolynomial}.
See Remark \ref{proofofexistenceofminimalpolynomial} and \ref{proofofabovecriterion}.

\section{A degree upper bound for minimal polynomials}\label{A degree upper bound for minimal polynomials}
From the very beginning of the study of minimal polynomials, linear programming has been one of the approaches to computing the minimal polynomial of given $S$ (\cite{firstminimalpolynomial}, \cite{Thielcke}).
The idea is simple; suppose we are given $S=\{(a_i,b_i)\}_{i=1}^n \subset \mathbb{Q}_{>0} \times \Qnn$, and it has the minimal polynomial with degree $d$.
Then it suffices to solve the following equation, minimizing $x_d, \ldots, x_0$ inductively in this order:
\begin{align*}
  \begin{bmatrix}
a_{1}^{0} & \cdots & a_{1}^{d} \\
& \vdots & \\
a_{n}^{0} & \cdots & a_{n}^{d} 
\end{bmatrix}
  \begin{bmatrix}
x_{0} \\
 \vdots \\
x_{d} \\
\end{bmatrix}
&=  \begin{bmatrix}
b_{1} \\
 \vdots \\
b_{n} \\
\end{bmatrix}, \\
x_{0}, \ldots, x_{d} &\geq 0. 
\end{align*}

Therefore, if we are just given $S$, the degree upper bounds for the minimal polynomial of $S$ (if it exists) play a prominent role in designing algorithms and measuring their complexity.
Theorem \ref{discreteupperbound} in this section tackles this issue.

We begin with setting up our convention of the complexity of descriptions of rational numbers.
\begin{defi}
 Define $||{n}|| := \lceil \log_{2}(|n|+1) \rceil$ for $n \in \bb{Z}$.
\end{defi} 
In other words, $||{n}||$ is the length of the binary notation of $n$ (ignoring signs).
\begin{defi}
For $x=\frac{q}{p} \in \bb{Q}$, where $p>0$ and $p,q$ are coprime, set $\sz(x) := \max\{||{p}||, ||{q}||\}$.
\end{defi}

Example \ref{baseexample} tells us that $\deg(f)$ for the minimal polynomial $f$ of given $S$ may be large. 
However, by the definition of minimal polynomial, $f$ is at lease $\# S$-sparse, that is, the number of monomials in $f$ with non-zero coefficients should be at most $\# S$.
Therefore, the following degree upper bound for sparse interpolation is instructive:
\begin{thm}[\cite{degreeupperboundforsparseinterpolation}]\label{degreeupperboundforsparseinterpolation}
Let $t \geq 1$, $S=\{(x_i,y_i)\}_{i=1}^{t+1} \subseteq \mathbb{Q}_{>0} \times \mathbb{Q}$.
Let $f(X)$ be a $t$-sparse polynomial (that is, the number of monomials with non-zero coefficients is at most $t$) in $\mathbb{Q}[X]$ interpolating $S$.
Then
\[ \deg (f) \leq \left(t+\log_{\gamma}[(t!)2^{t+1}\lambda \beta (\alpha \lambda)^{t^2}]\right)\left(1+\log_{\gamma}[\alpha \lambda]\right)^{t-1},\]
where $\alpha = \max_{i} x_i$, $\beta=\max_i |{y_i}|$, $\gamma = \{x_j/x_i \mid x_j > x_i\}$, and $\lambda$ is the least common multiple of the denominator of $x_i$'s and $y_i$'s in the reduced form. 
\end{thm}

Actually, taking a close look at the proof in \cite{degreeupperboundforsparseinterpolation}, we can observe that $x_i$'s and $y_i$'s contribute to the degree upper bound independently:
\begin{cor}\label{degreeupperboundseparateversion}
In the setting in Theorem \ref{degreeupperboundforsparseinterpolation}, we further have the following:
\[\deg (f) \leq \left(t+\log_{\gamma}[(t!)2^{t+1}\nu \beta (\alpha \mu)^{t^2}]\right)\left(1+\log_{\gamma}[\alpha \mu]\right)^{t-1},\]
where $\mu$ is the least common multiple of the denominator of $x_i$'s, and $\nu$ is that of $y_i$'s. 
\end{cor}

\begin{proof}[Sketch]
We just mention how we change the proof of Theorem 4.5 in \cite{degreeupperboundforsparseinterpolation}.
Immediately after the formula $(4.6)$ in \cite{degreeupperboundforsparseinterpolation}, they deduced $\det \mathbf{Z}^{i_1, \ldots, i_k}_{1,\ldots,k}=0$ or 
\[|{\det \mathbf{Z}^{i_1, \ldots, i_k}_{1,\ldots,k}}| \geq \left(\frac{1}{\lambda} \right)^{1+\sum_{j=1}^{k-1}D_j}.\]
We change this bounding term $\left(\frac{1}{\lambda} \right)^{1+\sum_{j=1}^{k-1}D_j}$ to $\frac{1}{\nu \mu^{\sum_{j=1}^{k-1}D_j}}$.
Then, the occurrences of the term $\lambda^{1+\sum_{i=1}^{k-1}D_i}$ in the latter inequalities are replaced with $\nu \mu^{\sum_{i=1}^{k-1}D_i}$, and the last inequality on page 80 of \cite{degreeupperboundforsparseinterpolation} becomes 
\[ \gamma^{(e_k-t)} \geq (t!) \binom{t+1}{k} \nu \mu^{\sum_{i=1}^{k-1}D_i}\frac{\prod_{i>j}(x_{b_i}-x_{b_j})}{\prod_{i>j}(x_{v_i}-x_{v_j})} \alpha^{\sum_{i=1}^{k-1}D_i}\beta.\]

Since each $(x_{b_i}-x_{b_j})$ is at most $\alpha$ and each $(x_{v_i}-x_{v_j})$ is at least $1/\mu$, we can change the subsequent recurrence in \cite{degreeupperboundforsparseinterpolation} to the following:
\[D_k = t+\log_{\gamma}\left\{ (t!)2^{t+1} \nu \beta (\alpha\mu)^{t^2 + \sum_{i=1}^{k-1}D_i} \right\}.\]

Thus, solving the recurrence in the same manner as \cite{degreeupperboundforsparseinterpolation}, we see the claim follows.
\end{proof}

Therefore, if $f$ is the minimal polynomial of $S=\{(a_i,b_i)\}_{i=1}^n \subset \mathbb{Q}_{>0} \times \Qnn$ and if $f$ is in addition $(n-1)$-sparse, then we have $\deg(f) \leq exp(n,A) poly(B)$, where $A= \max_i \sz(a_i)$, $B=\max_i \sz(b_i)$, by the following calculation:
\begin{align*}
&\left(n+\log_{\gamma}[(n!)2^{n+1}\nu \beta (\alpha \mu)^{n^2}]\right)\left(1+\log_{\gamma}[\alpha \mu]\right)^{n-1} \\
=& \left(n+\frac{\log_{2}[(n!)2^{n+1}\nu \beta (\alpha \mu)^{n^2}]}{\log_2 \gamma} \right)\left(1+\frac{\log_{2}[\alpha \mu]}{\log_2 \gamma}\right)^{n-1}\\
\leq & \left(n+\frac{poly(n,A,B)}{\log_2 \gamma} \right)\left(1+\frac{2A}{\log_2 \gamma}\right)^{n-1},\\
& \log_2 \gamma \geq \min_{j > i}\left\{1,(a_{j}/a_{i}-1)\right\} \geq 2^{-2A}.
\end{align*}
This gives a degree upper bound for the unique CMP interpolating given $S$ if it exists by Corollary \ref{uniqueCMPisminimalpolynomial}.
However, the minimal polynomial $f$ of $S$ can have $n$-terms, and this case is not covered by the theory of sparse interpolation in $\mathbb{Q}[X]$; if we do not stick to CMP's, then for any natural numbers $e_1 < \cdots < e_n$, there exists an $n$-sparse polynomial of the form $\sum_{j=1}^n c_j X^{e_j}$ interpolating $S$.
In other words, it is trivial to find $\# n$-sparse interpolating given $n$ points, and we cannot hope degree upper bounds for it.

Still, we can obtain an upper bound of the same order $exp(n,A) poly(B)$ for the case when $\mathfrak{d}(f)$, $n$, and the number of terms of $f$ coincide.

Towards the bound, we introduce the following operation on sign sequences:
\begin{defi}
Let $s\in \Sp$ be a sign sequence.
Divide $\supp(s)$ into the intervals $X_0,\ldots,X_k$ in the manner presented in Proposition \ref{rmk:mucalculation}.
\begin{enumerate}
\renewcommand{\labelenumi}{(\theenumi)}
\item For each $i \in [k]$, put
\[
\floor{X_i}:=\begin{cases}
\{\min X_i-1\}\amalg X_i& (\text{$\#X_i$ : odd})\\
X_i&(\text{otherwise})\text.
\end{cases}
\]
$\floor{s}\in \Sp$ is defined by
\[
\supp(\floor{s})={X_0} \amalg \floor{X_{1}}\amalg \cdots\amalg \floor{X_k}.
\]
\item For each $i \in [k]$, put
\[
\ceil{X_i}:=\begin{cases}
X_i\amalg \{\max X_i+1\}& (\text{$\#X_i$ : odd})\\
X_i&(\text{otherwise})\text.
\end{cases}
\]
$\ceil{s}\in \Sp$ is defined by
\[
\supp(\ceil{s})={X_0}\amalg \ceil{X_{1}} \amalg \cdots\amalg \ceil{X_k}.
\]
\end{enumerate}
\end{defi}
\begin{eg}
\begin{enumerate}
\renewcommand{\labelenumi}{(\theenumi)}
\item
If $s=\underbrace{\plus}_{l_0}\zero\zero\underbrace{\plus}_{\text{$l_1$ : odd}}\zero\underbrace{\plus\plus}_{\text{$l_2$ : even}}$,
then
\begin{align*}
\floor{s}&=\plus\zero\plus\plus\zero\plus\plus,\\
\ceil{s}&=\plus\zero\zero\plus\plus\plus\plus.
\end{align*}
\item
If $s=\underbrace{}_{l_0}\zero\underbrace{\plus\plus\plus}_{\text{$l_1$ : odd}}\zero\zero\underbrace{\plus}_{\text{$l_2$ : odd}}$,
then
\begin{align*}
\floor{s}&=\plus\plus\plus\plus\zero\plus\plus,\\
\ceil{s}&=\zero\plus\plus\plus\plus\zero\plus\plus.
\end{align*}
\end{enumerate}
\end{eg}

The following observation is useful when we manipulate interpolation CMPs:

\begin{lemm}\label{survivingcoefficients}
Let $s \in \Sp$.
\begin{enumerate}
\renewcommand{\labelenumi}{(\theenumi)}
\item Let $\supp(\floor{s})=\{d_1>\cdots >d_m\}$.
We have $d_{2i+1}\in \supp(s)$ for $2i+1 \in  [m]$.
\item Let $\supp(\ceil{s})=\{d_1>\cdots >d_m\}$.
We have $d_{2i}\in \supp(s)$ for $2i\in [m]$.
\end{enumerate}
\end{lemm}

The following bound allows us to take an inductive approach:
\begin{thm}\label{inductiveupperbound}
Let $n \geq 1$. 
Assume $f$ be the minimal polynomial of $S=\{(a_{i},b_{i})\}_{i=1}^{n+1} \subset \bb{R}_{>0} \times \Rnn$ $(0<a_{1}<\cdots<a_{n+1})$, and
  $S_{n}:=\{(a_{i},b_{i})\}_{i=1}^{n}$ has a minimal polynomial $f_{0} \in \Rnn[X]$.
Furthermore, assume $\mathfrak{d}(f_{0})=n$ and $f_{0} \neq f$.
Let $s := \Sign(f_{0})$ and $\supp(\lfloor s \rfloor)=\{ d_{1} >  \cdots > d_{n}\}$. (Note that $\#\supp(\lfloor s \rfloor) = \mathfrak{d}(f_{0})$.)

Define the $(n+1) \times n$-matrix $M$ and the $(n+1)$-dimensional vector $\mathbf{b}$ by
\[M := \begin{bmatrix}
 a_{1}^{d_{n}} & \cdots & a_{1}^{d_{1}}\\
   \vdots & & \vdots \\
  a_{n+1}^{d_{n}} & \cdots & a_{n+1}^{d_{1}}\\
\end{bmatrix}
\quad \text{and} \quad
\mathbf{b} := \begin{bmatrix}
    b_1 \\
    \vdots \\
    b_{n+1}
\end{bmatrix}.
\]
For each $j \in [n]$, let $\Delta_{j}$ be the $(n+1) \times n$-matrix obtained from $M$ by replacing the $(n-j+1)$-th column, which consists of $a_i^{d_j}$, with $\mathbf{b}$.
For $i \in [n+1]$, let $\Delta_{j}^{(i)}$ be the $(n \times n)$-matrix obtained by removing the $i$-th row of $\Delta_{j}$.
Let $D_{j}^{(i)} := \det\Delta_{j}^{(i)}$.

Then, for each odd $j \in [n]$ and any $i \in [n+1]$, $D_{j}^{(i)} > 0$.
Furthermore, if $k$ is the least natural number satisfying $k > d_1$ and
\[ k \geq \frac{\log(nD_{j}^{(i)}/D_{j}^{(n+1)})}{\log(a_{n+1}/a_{i})}\]
for any odd $j \in [n]$ and $i \in [n]$,
then we have $\deg(f) \leq k$.
\end{thm}

\begin{proof}
Because $f_{0} \neq f$, we have $f_{0}(a_{n+1}) < b_{n+1}$.
Let 
\[f_{0}(X) = c_{d_{1}}X^{d_{1}} + \cdots + c_{d_{n}}X^{d_{n}}.\]

Fix an odd $j \in [n]$ and an arbitrary $i \in [n+1]$.
First we prove $D_{j}^{(i)} > 0$.
Let $M_{i}$ be the $n \times n$-matrix obtained by removing the $i$-th row from the $(n+1) \times n$-matrix
$M$.
We observe that $b_l = f_0(a_l) = c_{d_{1}}a_l^{d_{1}} + \cdots + c_{d_{n}}a_l^{d_{n}}$ for $l \in [n]$.
Using elementary column operations on $\Delta_j^{(i)}$ and the multilinearity of the determinant, we see that
\[D_j^{(i)} = \begin{cases}
    c_{d_j} \det M_{n+1} & (i = n+1) \\
    c_{d_j} \det M_i + {} \\
    \quad (b_{n+1} - f_0(a_{n+1})) \det M_{i,j} & (i \in [n]),
\end{cases}
\]
where $M_{i,j}$ is the $(n-1) \times (n-1)$-matrix obtained from $M_{i}$ by deleting the bottom row and the $(n-j+1)$-th column.
We get $c_{d_{j}} > 0$ by Lemma \ref{survivingcoefficients}.
We have $\det M_{i}, \det M_{i,j} > 0$ since they are minors of a Vandermonde matrix with positive entries.
Thus, we obtain $D^{(i)}_{j} > 0$.

Next, we show $\deg(f) \leq k$.
By the minimality of $f$, it is enough to show that there exists $g \in \Rnn[X]$ interpolating $S$ and satisfying $\deg(g) \leq k$.

We search $g$ of the following form:
\[g = \sum_{i=1}^{n}\alpha_{d_{i}} X^{d_{i}} + \alpha_{k} X^{k}.\]
We obtain $g \in \bb{R}[X]$ interpolating $S$ by solving the linear equation
\[\Gamma [\alpha_{d_n}, \ldots, \alpha_{d_1}, \alpha_k]^t = \mathbf{b},\]
where
$\mathbf{a}^k:=\left[a_{1}^k, \ldots, a_{n+1}^k\right]^{t}$ and $\Gamma := \left[ M \, \mathbf{a}^k \right]$.
Hence, the remaining task is to show that all $\alpha_{d_{i}}$ and $\alpha_{k}$ of the solution are nonnegative.
We have $\det\Gamma > 0$ since it is a minor of a Vandermonde matrix of positive entries.
By Cramer's formula, we obtain $\alpha_{d_{j}}=E_j/\det\Gamma$ and $\alpha_k = E' / \det \Gamma$, where
$E_j := \det \left[\Delta_j \ \mathbf{a}^k \right]$ and $E' := \det [M \, \mathbf{b}]$.
Hence, it suffices to prove $E_j \ge 0$ for $j \in [n]$ and $E' \ge 0$.
By reasoning similar to that for $D_j^{(i)}$ above, we deduce that $E_j > 0$ for an even $j \in [n]$ and $E' > 0$.

Now, consider odd $j \in [n]$.
Taking a cofactor expansion, we get
\begin{align*}
 E_j&= D_{j}^{(n+1)}a_{n+1}^{k} + \sum_{i=1}^{n} (-1)^{n+1+i}D_{j}^{(i)}a_{i}^{k}\\
 &= D_{j}^{(n+1)}a_{n+1}^{k}  \left( 1 + \sum_{i=1}^{n} (-1)^{n+1+i}\frac{D_{j}^{(i)}}{D_{j}^{(n+1)}}\left(\frac{a_{i}}{a_{n+1}}\right)^{k}\right).
\end{align*}
By the definition of $k$, each of the $n$-terms in the sum is bounded by $1/n$, and hence the result follows.
\end{proof}

\begin{rmk}\label{proofofexistenceofminimalpolynomial}
    Note that the proof above and Theorem \ref{existenceofminimalpolynomial} imply Theorem \ref{abovecriterion}.
\end{rmk}

\begin{thm}\label{discreteupperbound}
Let $n \geq 2$, $S=\{(a_{i},b_{i})\}_{i=1}^{n}$ ($0<a_{1}<\cdots<a_{n}$, and $a_{i}, b_{i} \in \bb{Q}$ for each $i \in [n]$), and assume $f \in \Qnn[X]$ is the minimal polynomial of $S$.
Let $\alpha, \beta, \gamma, \mu, \nu$ be the parameters as in Theorem \ref{degreeupperboundforsparseinterpolation} and Corollary \ref{degreeupperboundseparateversion}, that is,
$\alpha = \max_i a_i = a_n$, $\beta = \max_i b_i=b_n$, $\gamma=\min \{a_j / a_i \mid a_j > a_i\}$, $\mu$ is the least common multiple of the denominators of $a_i$'s and $\nu$ is that of $b_i$'s.
Then 
\begin{align*}
\deg(f) \leq (n-1)! \log_{\gamma}\left((n+1)!\beta \nu\right)\left(\log_{\gamma}(\alpha \mu)\right)^{n-2}
\end{align*}

\end{thm}

\begin{proof}
First, note that $\beta \nu \geq 1$ and $\alpha \mu \geq \gamma > 1$, and therefore $\log_{\gamma} (\beta \nu) \geq 0$ and $\log_{\gamma} (\alpha \mu)\geq 1$.
Let $S_{m}:=\{(a_{i},b_{i})\}_{i=1}^{m}$ for $m \in [n]$, and let $f_{m}$ be the minimal polynomial of $S_{m}$.
Let 
\[\delta_m := (m-1)! \log_{\gamma}\left((m+1)!\beta \nu\right)\left(\log_{\gamma}(\alpha\mu)\right)^{m-2}.\]
By induction on $m \geq 1$, we prove 
\[\deg(f_{m}) \leq \delta_{m}.\] 
Then the case $m=n$ gives the claim.

The base case when $m=1$ is clear since $\deg(f_{1})\leq 0$.
Assume the statement holds for $m$ and consider $m+1$.
If $f_{m+1}=f_{m}$, then the statement is immediate from the induction hypothesis, so we focus on the case when $f_{m+1}\neq f_{m}$.
We use the notations in Theorem \ref{inductiveupperbound} applied to $f =
f_{m+1}$, $S=S_{m+1}$, $f_{0} = f_{m}$, $S_{n}=S_{m}$.
Note that $d_{1}$ there equals to $\deg(f_{m})$ in this context.
It is enough to show that
\begin{align}
\frac{\log\left(m\left(D_{j}^{(i)}/D_{j}^{(m+1)}\right)\right)}{\log(a_{m+1}/a_{i})} \leq \delta_{m+1} \label{aim}
\end{align}
for any odd $j \in [m]$ and arbitrary $i \in [m]$.

Indeed, we observe
\[ \log \gamma \leq \log(a_{m+1}/a_i).\]
Furthermore, expanding $D_{j}^{(i)}$ to a sum of $m!$-many products of $(m-1)$ terms of the form $a_{i}^{d_{j}}$ and one term of the form $b_{l}$, we obtain
\[D^{(i)}_j \leq m! \alpha^{d_1(m-1)} \beta.\]
On the other hand, expanding $D_{j}^{(m+1)}$ similarly as $D_{j}^{(i)}$ above, we have $(\mu^{d_1 (m-1)}\nu)D_{j}^{(m+1)} \in \mathbb{Z}$. 
Knowing $D_{j}^{(m+1)} > 0$, we have $D_{j}^{(m+1)} \geq (\mu^{d_1 (m-1)}\nu)^{-1}$.
Thus, we can bound the LHS of (\ref{aim}) as follows:
\begin{align*}
\frac{\log\left(m\left(D_{j}^{(i)}/D_{j}^{(m+1)}\right)\right)}{\log(a_{m+1}/a_{i})} 
&\leq \frac{\log\left(m\times m! \alpha^{d_1(m-1)} \beta \times \mu^{d_1 (m-1)}\nu\right)}{\log \gamma } \\
&\leq \log_{\gamma}\left( (m+1)! (\alpha \mu)^{d_1(m-1)} \beta \nu \right) \\
&\leq d_1 (m-1) \log_{\gamma}(\alpha \mu) + \log_{\gamma}\left((m+1)!\beta \nu\right).
\end{align*}

By induction hypothesis, $d_1 \leq \delta_m$.
Therefore, we can bound the last bound further:
\begin{align*}
&d_1 (m-1) \log_{\gamma}(\alpha \mu) + \log_{\gamma}\left((m+1)!\beta \nu\right)\\
\leq& (m-1)! \log_{\gamma}\left((m+1)!\beta \nu\right)\left(\log_{\gamma}(\alpha\mu)\right)^{m-2} \times (m-1) \log_{\gamma}(\alpha \mu) + \log_{\gamma}\left((m+1)!\beta \nu\right)\\
\leq& m!\left( \log_{\gamma}(\alpha \mu) \right)^{m-1}\log_{\gamma}((m+1)!\beta \nu)\\
\leq& \delta_{m+1}.
\end{align*}
As for the second last inequality, note that $\log_{\gamma}(\alpha \mu) \geq 1$ since $a_j/a_i \leq \alpha \mu$.
\end{proof}

For the discussion on whether the above degree upper bound is optimal or not, see section \ref{Open Problems}. 


As a quick application of Theorem \ref{discreteupperbound}, we consider the complexity of a straightforward algorithm using linear programming to compute the minimal polynomial $f$ of given $S$: Algorithm \ref{minimalpolynomialbylinearprogramming}.
Such algorithms are ones of the earliest in this field (\cite{firstminimalpolynomial}, \cite{Thielcke}).

\begin{figure}[t]
  \begin{algorithm}[H]
    \caption{Minimal Polynomial by Linear Programming}
    \label{minimalpolynomialbylinearprogramming}
    \begin{algorithmic}[1]
      \Require $n \geq 1$, $S=\{(a_i,b_i)\}_{i=1}^n \subset \bb{Q}_{>0} \times \Qnn$ with $a_{1}<\cdots<a_{n}$.
      \Ensure The minimal polynomial of $S$ if it exists. Otherwise, ``none."
      \State $U \leftarrow$ the degree upper bound given in Theorem \ref{discreteupperbound}
      \State $q \leftarrow 1$
      \State $r \leftarrow$ ``not yet"
      \While{$r=$ ``not yet"}\label{whilecondition}
      \State\label{solvelinearprogramming} Solve the following linear programming:
      \begin{align*}
  \begin{bmatrix}
a_{1}^{0} & \cdots & a_{1}^{q} \\
& \vdots & \\
a_{n}^{0} & \cdots & a_{n}^{q} 
\end{bmatrix}
  \begin{bmatrix}
x_{0} \\
 \vdots \\
x_{q} \\
\end{bmatrix}
&=  \begin{bmatrix}
b_{1} \\
 \vdots \\
b_{n} \\
\end{bmatrix}, \\
x_{0}, \ldots, x_{q} &\geq 0. 
 \end{align*}
 \If{such $[x_0, \ldots, x_q]^t$ exists}
 \State\label{degreeofminimalpolynomial} $d \leftarrow q$
 \State\label{found} $r \leftarrow$ ``found"
 \ElsIf{$q=U$}
 \State {\bf return} ``none"
 \Else \Comment{The case when $q<U$}
 \State\label{endofwhile} $q \leftarrow q+1$
 \EndIf
      \EndWhile
 \State\label{minimizetheleading} {\bf minimize} $y_d$ {\bf subject to}:
 \begin{align*}
  \begin{bmatrix}
a_{1}^{0} & \cdots & a_{1}^{d} \\
& \vdots & \\
a_{n}^{0} & \cdots & a_{n}^{d} 
\end{bmatrix}
  \begin{bmatrix}
y_{0} \\
 \vdots \\
y_{d} \\
\end{bmatrix}
&=  \begin{bmatrix}
b_{1} \\
 \vdots \\
b_{n} \\
\end{bmatrix}, \\
y_{0}, \ldots, y_{d} &\geq 0\text.
 \end{align*}
\State {\bf return} $y_{d}X^d + \cdots + y_{0} X^0 \in \Qnn[X]$
    \end{algorithmic}
  \end{algorithm}
\end{figure}
\begin{prop}
    Algorithm \ref{minimalpolynomialbylinearprogramming} outputs the minimal polynomial $f \in \bb{Q}_{\geq 0}[X]$ of $S$ if it exists and ``none'' if there is none. 
\end{prop}
\begin{proof}
Clearly, Algorithm \ref{minimalpolynomialbylinearprogramming} returns ``none'' if there does not exist a minimal polynomial of $S$. 

Consider the case when $S$ has a minimal polynomial $f$.
Then there are $\deg(f)$-times of repetition of solving linear programming during lines \ref{whilecondition}-\ref{endofwhile}, and it finds $\deg(f)$ as $d$ in line \ref{degreeofminimalpolynomial} since Theorem \ref{equivalence} (\ref{strongminimality}) implies that $\deg(f)$ is the minimum of degrees of CMPs interpolating $S$.

The latter line \ref{minimizetheleading} finds an interpolating CMP $g$ of degree $d$ whose leading coefficient attains the minimum, and actually $g=f$.
Indeed, if $g \neq f$, then $f-g$ has $a_1, \ldots, a_n$ as their roots. 
Since 
\[\SC(\Sign(f-g)) \leq \mathfrak{d}(f) \leq n,\]
$\SC(\Sign(f-g))=n=\mathfrak{d}(f)$ by Descartes' rule of signs. 
Hence, by Lemma \ref{achieve} and $\deg(f)=\deg(g)$, the leading coefficient of $g$ is greater than that of $f$, which is a contradiction.

Thus the algorithm outputs the minimal polynomial of $S$.
\end{proof}

\begin{rmk}\label{proofofabovecriterion}
    Note that the argument above, verifying that the output polynomial is the minimal polynomial, also gives a proof of Theorem \ref{existenceofminimalpolynomial} (cf. \cite{Thielcke}).
\end{rmk}

\begin{thm}\label{thm:complexity-of-algo-by-LP}
    Let $U$ be the upper bound in Theorem \ref{discreteupperbound}.
    Let $A= \max_i \sz(a_i)$ and $B=\max_i \sz(b_i)$.
    Algorithm \ref{findtheminimalpolynomial} runs in $\poly(U)$-time. 
    In particular, it runs in $\expt(n,A)\poly(B)$-time.
\end{thm}
\begin{proof}
Algorithm \ref{minimalpolynomialbylinearprogramming} executes linear programming $O(U)$-times repeatedly. 
Each time of the repetition during lines \ref{whilecondition}-\ref{endofwhile}, the size of the input to the linear programming is bounded above by
\[ O(dA)\times n(d+1) + nB \leq \poly(n,A,B,d).\]
Linear programming can be solved in polynomial time, so the whole costs $\poly(n,A,B,d)$-time.
The last optimization step \ref{minimizetheleading} is another linear programming, and this is solved in $\poly(n,A,B,d)$-time by a similar reasoning.


Hence, the whole procedure takes at most $\poly(n,A,B,d) \leq \poly(U)$-time (note that $\poly(n,A,B) = O(U)$).

In the case when $S$ does not have a minimal polynomial, the algorithm also runs in $\poly(U)$-time (equivalent to the case when $d=U$ above), and this completes the analysis.
\end{proof}

\section{An order structure of sign sequences}\label{An order structure of sign sequences}
In this section, we introduce a useful order structure on $\Sp$ which neatly expresses our manipulation of polynomials in the algorithm presented in section \ref{An algorithm}.

We begin by considering the following subset of $\Sp$:

\begin{defi}
Set 
\[
\Spc:=\{s\in \Sp\mid \mathfrak{d}(s)=\#\supp(s)\},
\]
that is, for $s \in \Sp$, $s \in \Spc$ if and only if each $l_{i}$ $(i \in [k])$ described in Proposition \ref{rmk:mucalculation} is even.
\end{defi}

\begin{rmk}
 We point out that $\{\supp(s) \mid s \in \Spc\}$ already appeared in \cite[Theorem 3]{Interpolationbypolynomialswithnonnegativecoefficients} in the characterization of the set of vertices of a clam giving surface hyperplanes.
 See the last part of \S \ref{An algorithm} for the relation between \cite{Interpolationbypolynomialswithnonnegativecoefficients} and our paper.
\end{rmk}

The following is immediate from the definitions:
\begin{lemm}\label{firstthingsonSpc}
For $s\in \Sp$, the following hold:
\begin{enumerate}
\renewcommand{\labelenumi}{(\theenumi)}
\item $\floor{s}, \ceil{s}\in \Spc$.
\item $\mathfrak{d}(s)=\mathfrak{d}(\floor{s})=\mathfrak{d}(\ceil{s})$.
\item $s\in \Spc \Longleftrightarrow \floor{s}=\ceil{s}\Longleftrightarrow \floor{s}=s=\ceil{s}$.
\end{enumerate}
\end{lemm}
Now, we introduce an order $\prec_0$ on $\Spc$:
\begin{defi}[Order structure on $\Spc$]
Let $s,s^{\prime}\in \Spc$, and write
\[
\supp(s)=\{d_1>\cdots>d_m\},\ \supp(s^{\prime})=\{d'_1>\cdots>d'_{m'}\}.
\]
The order structure $\preceq_0$ on $\Spc$ is defined by
\begin{align*}
s\preceq_0 s^{\prime}:\Iff
m\leq m' \ \& \
d_i\leq d_i' \ (i\in [m]).
\end{align*}
\end{defi}

\begin{rmk}\label{homomorphism}
 By definition, for $s,s^{\prime}\in \Spc$, it follows that if $s \preceq_0 s^{\prime}$, then $\mathfrak{d}(s) \leq \mathfrak{d}(s^{\prime})$.
\end{rmk}

The following is immediate:
\begin{lemm}\label{conservation}
For $s\in \Sp$, $\floor{s}\preceq_0 \ceil{s}$.
\end{lemm}

Next, we extend the order on $\Spc$ to that of $\Sp$.

\begin{defi}[Order structure on $\Sp$]
The order structure $\preceq$ on $\Sp$ is defined by
\[
s\preceq s^{\prime}:\Iff
\text{$s = s^{\prime}$ or ($s\neq s^{\prime}$ \& $\ceil{s}\preceq_0 \floor{s^{\prime}}$).}
\]
\end{defi}
It is clear that $(\Sp,\preceq)$ is a poset extending $(\Spc, \preceq_0)$ by Lemma \ref{conservation}.
\begin{eg}
\[
\zero \prec \plus \prec \plus\plus \prec \zero\plus \prec \zero\plus\plus \prec \zero\zero\plus.
\]

For further examples, see the diagram below:
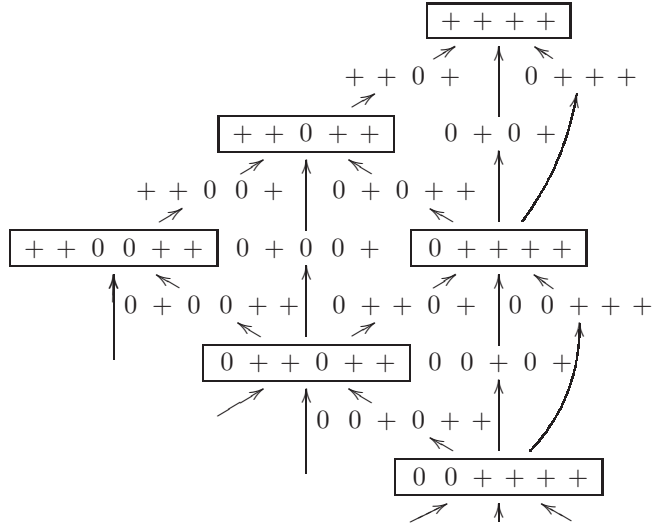
\[
\xymatrix@C=-45pt@R=5pt{
&&&&&\fbox{\plus\plus\plus\plus}\ar@{<-}[ld]\ar@{<-}[dd]\ar@{<-}[rd]&\\
&&&&\plus\plus\zero\plus\ar@{<-}[ld]&&\zero\plus\plus\plus\ar@{<-}@/^/[lddd]\\
&&&\fbox{\plus\plus\zero\plus\plus}\ar@{<-}[ld]\ar@{<-}[dd]\ar@{<-}[rd]&&\zero\plus\zero\plus\ar@{<-}[dd]&\\
&&\plus\plus\zero\zero\plus\ar@{<-}[ld]&&\zero\plus\zero\plus\plus\ar@{<-}[rd]&&\\
&\fbox{\plus\plus\zero\zero\plus\plus}\ar@{<-}[ld]\ar@{<-}[dd]\ar@{<-}[rd]&&\zero\plus\zero\zero\plus\ar@{<-}[dd]&&\fbox{\zero\plus\plus\plus\plus}\ar@{<-}[ld]\ar@{<-}[dd]\ar@{<-}[rd]&\\
&&\zero\plus\zero\zero\plus\plus\ar@{<-}[rd]&&\zero\plus\plus\zero\plus\ar@{<-}[ld]&&\zero\zero\plus\plus\plus\ar@{<-}@/^10pt/[lddd]\\
&&&\fbox{\zero\plus\plus\zero\plus\plus}\ar@{<-}[ld]\ar@{<-}[dd]\ar@{<-}[rd]&&\zero\zero\plus\zero\plus\ar@{<-}[dd]&\\
&&&&\zero\zero\plus\zero\plus\plus\ar@{<-}[rd]&&\\
&&&&&\fbox{\zero\zero\plus\plus\plus\plus}\ar@{<-}[ld]\ar@{<-}[d]\ar@{<-}[rd]&\\
&&&&&&
}
\]
Above, we write $s^{\prime} \leftarrow s$ if $s^{\prime} \prec s$, and the elements of $\Spc$ are in framed boxes.

\end{eg}

\begin{lemm}\label{generalhom}
For $s,s^{\prime} \in \Sp$, the following hold:
\begin{enumerate}
\renewcommand{\labelenumi}{(\theenumi)}
 \item $\floor{s} \preceq s \preceq \ceil{s}$.
 \item If $s \preceq s^{\prime}$, then $\floor{s} \preceq \floor{s^{\prime}}$, $\ceil{s} \preceq \ceil{s^{\prime}}$ and $\mathfrak{d}(s) \preceq \mathfrak{d} (s^{\prime})$.
\end{enumerate} 
\end{lemm}

\begin{proof}
 $\floor{s} \preceq s$ follows from $\ceil{\floor{s}} = \floor{s}$ and $s \preceq \ceil{s}$ follows from $\ceil{s} = \floor{\ceil{s}}$.
 
 Now let $s \prec s^{\prime}$.
  $\floor{s} \preceq \floor{s^{\prime}}$ follows from $\floor{s} \preceq \ceil{s} \preceq \floor{s^{\prime}}$,
  and $\ceil{s} \preceq \ceil{s^{\prime}}$ can be shown similarly. 
  $\mathfrak{d}(s) \preceq \mathfrak{d} (s^{\prime})$ also follows from
  \[\mathfrak{d}(s) = \mathfrak{d} (\ceil{s}) \preceq \mathfrak{d} (\floor{s^{\prime}}) = \mathfrak{d}(s^{\prime}).\]
  Here, we have used Remark \ref{homomorphism} and Lemma \ref{conservation}.
\end{proof}

\begin{lemm}\label{well-founded}
\begin{enumerate}
 \renewcommand{\labelenumi}{(\theenumi)}
 \item\label{eventuallyincreasing} If $s_0\prec s_1 \prec s_2$ in $\Sp$, then $\floor{s_0} \prec \floor{s_2}$.
\item\label{well-foundedness} $(\Sp, \prec)$ is well-founded.
\item\label{degreediverge} If $s_0 \prec s_1 \prec s_2 \prec \cdots\ (s_i\in \Sp)$, then $\lim_{i\rightarrow \infty }\deg(s_i)=\infty$.
\item\label{numberofiterations} Let $s_1 \prec \cdots \prec s_l$ be a chain in $\Spc$, and $\#supp(s_l) = n \geq 2$. Then $l \leq n\deg(s_l)$.
\end{enumerate}
\end{lemm}

\begin{proof}
We first consider (\ref{eventuallyincreasing}).
Indeed,
$\floor{s_0}\preceq \ceil{s_0} \preceq \floor{s_1} \preceq \ceil{s_1} \preceq \floor{s_2}$, hence, if $\floor{s_0}=\floor{s_2}$, then  $s_0=s_1\in \Spc$.
(\ref{well-foundedness}) follows immediately from (\ref{eventuallyincreasing}) and the fact that $(\Spc, \preceq_0)$ is well-founded.
(\ref{degreediverge}) follows from (\ref{eventuallyincreasing}) and the definition of $\preceq_0$.
We prove (\ref{numberofiterations}).
For each $s \in \Spc$, let 
$\mu(s) := \sum_{d \in \supp(s)} d$.
Then we have $0 \leq \mu(s_1) < \cdots < \mu(s_l) < n \deg(s_l)$ by the definition of $\preceq_0$ and the assumption $n \geq 2$.
\end{proof}

The remaining lemmas in this section are immediate consequences of the definitions:
\begin{lemm}\label{increment}
Let $s\in \Spc$ and $\supp(s)=\{d_1>\cdots>d_m\}$.
Assume that $s^{\prime}\in \Sp$ satisfies
\[
\begin{cases}
\supp(s^{\prime})\subsetneq \supp(s), \\
\supp(s) \setminus \supp(s^{\prime})\subseteq \{d_{2i}\mid 1\leq 2i \leq m\}.
\end{cases}
\]
Then $s^{\prime}\notin \Spc$ and $\floor{s^{\prime}}=s$.
In particular, $s \prec s^{\prime}$ and $\mathfrak{d}(s^{\prime})=\mathfrak{d}(s)$.
\end{lemm}

\begin{defi}
Let $s\in \Spc$.
Put $d=\min\{i\mid s_i=0\}$.
Define $s^+$ by $\supp(s^+)=\supp(s)\cup\{d\}$.
\end{defi}

\begin{lemm}
Let $s\in \Spc$. 
Then $s^+\in \Spc$, $s \prec s^+$, and $\mathfrak{d}(s^+)=\mathfrak{d}(s)+1$.
\end{lemm}

\begin{lemm}\label{survivingcoefficientsforceil}
Let $s\in \Sp$ and 
$\supp(\ceil{s}^+)=\{d_1>\cdots >d_m\}$.
Then $d_{2i}\in \supp(s)$ for $2i\in [m]$.
\end{lemm}

\section{An algorithm along $\preceq$}\label{An algorithm}
\quad \cite{Thielcke} gave a linear-programming-free heuristics for finding a minimal polynomial.
However, as far as we see, there has not been a complete analysis of it, and natural trials actually would be complicated.

Thus we take a different approach. In this section, we describe another linear-programming-free algorithm for constructing the minimal polynomial of given $S=\{(a_{i},b_{i})\}_{i=1}^{n}$ ``along'' the order $\preceq$.

Although the resulting Algorithm \ref{findtheminimalpolynomial} itself is of performance equivalent to Algorithm \ref{minimalpolynomialbylinearprogramming}
in the sense that they both run in
$\expt(n,A)\poly(B)$-time, we think it is worth mentioning since it reveals some new facts on the relation between minimal polynomials (and interpolating CMPs) and may contribute to future solutions to some open problems in section \ref{Open Problems}.

We begin with a subroutine ``incrementing'' interpolating CMPs. 

\begin{figure}[t]
  \begin{algorithm}[H]
    \caption{Increment}
    \label{algorithmincrement}
    \begin{algorithmic}[1]
      \Require $f \in \Qnn[X]$, $m \geq 1$, and $a_{1}< \cdots < a_{m} \in \mathbb{Q}_{>0}$ with $\mathfrak{d}(f) = m, m+1$.
      \Ensure $g \in \Qnn[X]$ (see Corollary \ref{inductivestep}).
      \If{$\mathfrak{d}(f) = m$}
      \State $s \leftarrow \ceil{\Sign(f)}^{+}$.
      \Else \Comment{The case when $\mathfrak{d}(f) = m+1$.}
      \State $s \leftarrow \ceil{\Sign(f)}$.
      \EndIf
   \State $\{d_1 > \cdots > d_{m+1}\} \leftarrow \supp(s)$ \Comment{Note that $\# \supp(s)=m+1$ in any case.}
   \State\label{findshift} Find one $[c_{1},\ldots, c_{m+1}]^{t}$ such that
      \[\begin{bmatrix}
 a_{1}^{d_{1}} & \cdots & a_{1}^{d_{m+1}}\\
   \vdots & & \vdots \\
  a_{m}^{d_{1}} & \cdots & a_{m}^{d_{m+1}}\\
\end{bmatrix}
\begin{bmatrix}
c_{1}\\
\vdots \\
c_{m+1}
\end{bmatrix}
=\mathbf{0}, c_{1} > 0.\]
      \State $h \leftarrow c_{1}X^{d_{1}} + \cdots + c_{m+1}X^{d_{m+1}}$\label{theshift} 
      \State Let $f=e_{1}X^{d_{1}} + \cdots + e_{m+1}X^{d_{m+1}}$.
      \State \Comment{Note that $\supp(\Sign(f)) \subseteq \supp(s)$.}
      \State $t \leftarrow \min_{i=1}^{\floor{(m+1)/2}}\left|{e_{2i}/{c_{2i}}}\right|$
      
      \State $g \leftarrow f+th$
      \State {\bf return } $g$
    \end{algorithmic}
  \end{algorithm}
\end{figure}

\begin{prop}\label{termination}
For Algorithm  \ref{algorithmincrement}, the system in line \ref{findshift} has a solution, and the algorithm outputs a CMP $g$.
\end{prop}
\begin{proof}
That the system in line \ref{findshift} has a solution follows from the optimality of Descartes' rule of signs.

Next, we show that $g \in \Qnn[X]$.
Consider $h$ in line \ref{theshift}, and let $\tau := \Sign(h)$.
By the argument above, $h(a_i)=0$ for $i \in [m]$, and  $\tau_{d_{i}}=(-1)^{i+1}$ for $i \in [m+1]$.

  Hence, we obtain $f+t h \in \bb{Q}_{\geq 0}[X]$ by Lemma \ref{survivingcoefficients} and \ref{survivingcoefficientsforceil}.
\end{proof}



By the analysis of $h$ above and the fact that $t$ is the maximum scalar such that $f+th \in \Qnn[X]$;
\begin{cor}\label{inductivestep}
For Algorithm \ref{algorithmincrement}, the following hold:
\begin{enumerate}
 \item $g(a_{i})=f(a_{i})$ for $i \in [m]$.
 \item\label{supportsize} Put $s^{\prime} := \Sign(g)$. The pair $(s,s^{\prime})$ satisfies the assumption of Lemma \ref{increment}.
 In particular, $\Sign(f) \preceq s \prec s^{\prime}$,
 \[\mathfrak{d}(s^{\prime})=\mathfrak{d}(s)=m+1,\ \mbox{and} \ \#\supp(s') \leq m.\]
 \item $g(a)>f(a)$ for $a > a_{m}$.
 \item $\Sign((1-u)f+ug)=s$ for $0<u<1$.
\end{enumerate}
\end{cor}
Now, we can describe our algorithm:
\begin{figure}[t]
  \begin{algorithm}[H]
    \caption{Find the minimal polynomial}
    \label{findtheminimalpolynomial}
    \begin{algorithmic}[1]
      \Require $n \geq 1$, $S=\{(a_i,b_i)\}_{i=1}^n \subset \bb{Q}_{>0} \times \Qnn$ with $a_{1}<\cdots<a_{n}$.
      \Ensure The minimal polynomial of $S$ if it exists. Otherwise, ``none."
      \State $g_1 \leftarrow b_1 \in \Qnn[X]$\label{initialg}
      \For{{$k=2, \ldots, n$}} \label{inductiveconstruction}
      \If{$b_k=g_{k-1}(a_k)$} \label{digression} 
      \State\label{formerg} $g_k \leftarrow g_{k-1}$.
      \ElsIf{$b_k < g_{k-1}(a_k)$ or $\mathfrak{d}(g_{k-1}) < k-1$}
      \State {\bf return} ``none" \Comment{We stop the computation in this case.} \label{nonecase} 
      \Else \Comment{The case when $g_{k-1}(a_k)<b_k$ and $\mathfrak{d}(g_{k-1}) = k-1$.}
      \State $g_{k,0} \leftarrow g_{k-1}$\label{abovecriterionapplies}
      \State $i \leftarrow 0$
      \While{$g_{k,i}(a_k) < b_k$} \label{iteratedincrements}
      \State\label{recursivecall} $g_{k,i+1} \leftarrow {\rm Increment}(g_{k,i}, k-1, a_1, \ldots, a_{k-1})$
      \State\label{pushthecounter} $i \leftarrow i+1$
      \EndWhile
      \State $t \leftarrow (b_k-g_{k,i-1}(a_k))/(g_{k,i}(a_k)-g_{k,i-1}(a_k))$
      \State$g_k \leftarrow (1-t)g_{k,i-1} + t g_{k,i}$\label{newg} 
     \EndIf
      \EndFor
      \State {\bf return} $g_n$
    \end{algorithmic}
  \end{algorithm}
\end{figure}
\begin{thm}
Algorithm \ref{findtheminimalpolynomial} outputs the minimal polynomial $f \in \bb{Q}_{\geq 0}[X]$ of $S$ if it exists and ``none'' if there is none. 
\end{thm}

\begin{proof}
For each $k \in [n]$, define $S_k := \{(a_i,b_i)\}_{i=1}^k$.

Inductively on $k \in [n]$, we show the following:
\begin{itemize}
 \item if $S_k$ has a minimal polynomial $h$, then $g_k$ is the one.
 \item otherwise, the algorithm returns ``none."
\end{itemize}

The base case $k=1$ is clear.
Assume that $k \geq 2$.
If $S_{k-1}$ does not have a minimal polynomial, then neither does $S_k$ by Theorem \ref{existenceofminimalpolynomial}, and the algorithm returns ``none" by induction hypothesis.
Suppose $S_{k-1}$ has a minimal polynomial, which is $g_{k-1}$ by induction hypothesis, but $S_k$ does not, then there are two cases by Corollary \ref{meaningofmu} and Theorem \ref{abovecriterion}:
\begin{itemize}
\item $\mathfrak{d}(g_{k-1})<k-1$ and $b_k \neq g_{k-1}(a_k)$.
\item $\mathfrak{d}(g_{k-1})=k-1$ and $b_k < g_{k-1}(a_k)$.
\end{itemize}

Hence, the algorithm returns ``none" by line \ref{nonecase}.
%

Assume that $S_{k-1}$ has a minimal polynomial, which is $g_{k-1}$, and that $S_k$ also has a minimal polynomial $h$.
If $\mathfrak{d}(g_{k-1})<k-1$, then $g_{k-1}=h$ and $g_{k-1}(a_k)=b_k$ by Corollary \ref{meaningofmu}, and $g_k=g_{k-1}=h$ by line \ref{formerg}.
If $\mathfrak{d}(g_{k-1})=k-1$, then we have $b_k \geq g_{k-1}(a_k)$ by Theorem \ref{abovecriterion}.
The case when $b_k = g_{k-1}(a_k)$ is similar to above.
We consider the case $b_k > g_{k-1}(a_k)$.
By Corollary \ref{inductivestep}, it holds that:
 \begin{align*}
     &\Sign(g_{k,0})<\Sign(g_{k,1}) < \cdots,\\
 &\mathfrak{d}(g_{k,0})=k-1, \quad \mathfrak{d}(g_{k,i}) = k, \quad (i =1,2,\ldots) \\
 &g_{k,0}(a_{k}) < g_{k,1}(a_{k}) < g_{k,2}(a_{k}) < \cdots .
 \end{align*}

  For the sake of contradiction, assume there is no critical $i$ which ends the \textbf{while}-loop in line \ref{iteratedincrements}.
  Since $\Sign(g_{k,i}) < \Sign(g_{k,i+1})$, we have
 $\deg(g_{k,i}) \rightarrow \infty$ as $i \rightarrow \infty$
by Lemma \ref{well-founded}.
 Consider sufficiently large $i$ such that $\deg(g_{k,i}) > \deg(h)$.
 Then we have $\SC(g_{k,i}-h) \leq k-1$ by
 $\mathfrak{d}(g_{k,i}) \leq k$ and Lemma \ref{achieve}.
 Hence, by Descartes' rule of signs, all the roots of $g_{k,i}-h$ are $a_{1}, \ldots, a_{k-1}$, which implies $g_{k,i}(a_{k}) > h(a_{k})=b_{k}$, a contradiction.
 Thus the algorithm finds a polynomial $g_k$.
 
 Lastly, $\mathfrak{d}(g_{k})\leq k$ by the definition of $g_{k}$ and Corollary \ref{inductivestep}.
\end{proof}

In order to consider the time complexity of the two algorithms, we measure the size of $f$ as follows:
\begin{defi}
 For $f:= c_1 X^{d_1} + \cdots + c_n X^{d_n} \in \bb{Q}[X]$, where $c_i \neq 0$ for $i \in [n]$, set $\sz(f) := \sum_{i=1}^n \sz(c_i)$.
\end{defi}

\begin{lemm}\label{complexityofalgorithmincrement}
    Algorithm \ref{algorithmincrement} runs in
    $\poly(m,\sz(f), \sum_{i=1}^m \sz(a_i),\deg(f))$-time. 
    Note that $m$ is redundant since $\sum_{i=1}^m \sz(a_i) \geq m$. We put it here for readability.
\end{lemm}
\begin{proof}
    Each line of the algorithm refers a polynomial time (with respect to the parameters above) function.
    Since there is no loop or iteration in Algorithm \ref{algorithmincrement}, it runs in polynomial time.
\end{proof}

\begin{thm}\label{complexityoffindtheminimalpolynomial}
    For Algorithm \ref{findtheminimalpolynomial}, let $M$ be the maximum index such that $S_M=\{(a_i,b_i)\}_{i=1}^M$ has a minimal polynomial $g_M$. 
    Let $d := \deg(g_M)$.
    Let $A= \max_i \sz(a_i)$ and $B=\max_i \sz(b_i)$.
    Then Algorithm \ref{findtheminimalpolynomial} runs in $\poly(n,A,B,d)$-time. 
    In particular, it runs in $\expt(n,A)\poly(B)$-time.
\end{thm}

\begin{proof}
    The dominant part of the algorithm is the iteration started in line \ref{inductiveconstruction}.
    The number of iterations is at most $n$.
    Each iteration includes a loop started in line \ref{iteratedincrements}.
    We have to be careful with line \ref{recursivecall}; the $(i+1)$-th step recursively calls the result of the $i$-th step, and also we need to bound the number of steps before the end of a loop.

    By Lemma \ref{well-founded} (\ref{eventuallyincreasing}) and (\ref{numberofiterations}), the number of steps we need is at most $O(dn)$.

    Despite (at most) polynomially many recursive applications of a polynomial-time algorithm, each $\sz(g_{k,i})$ can be bounded sufficiently; 
    first, by the monotonicity of $\Sign(g_{k,i})$ with respect to $<$, we have $\deg(g_{k,i}) \leq d$. 
    Furthermore, $\sz(g_{k,i}) \leq \poly(n,A,B,d)$.
    Indeed, $g_{k,i}$ is a result of Algorithm \ref{algorithmincrement}, thus
    $\#\supp(g_{k,i}) \leq k-1$ by Corollary \ref{inductivestep} (\ref{supportsize}).
    Hence, if 
    \[g_{k,i} = c_1 X^{d_1} + \cdots + c_{l}X^{d_l} \quad (d_1 < \cdots < d_l,\ l \leq k-1),\]
    then the coefficients form the solution of the following system:
     \[\begin{bmatrix}
 a_{1}^{d_{1}} & \cdots & a_{1}^{d_{l}}\\
   \vdots & & \vdots \\
  a_{l}^{d_{1}} & \cdots & a_{l}^{d_{l}}\\
\end{bmatrix}
\begin{bmatrix}
c_{1}\\
\vdots \\
c_{l}
\end{bmatrix}
=\begin{bmatrix}
b_{1}\\
\vdots \\
b_{l}
\end{bmatrix}.\]
Note that the coefficient matrix is a submatrix of a Vandermonde matrix, and therefore non-singular.

Now, the complexity of Algorithm \ref{findtheminimalpolynomial} follows from Lemma \ref{complexityofalgorithmincrement}.

The last stated complexity follows from Theorem \ref{discreteupperbound}.
\end{proof}

For $\{(a_i,b_i)\}_{i=1}^n$ as above, Bottema--Halpern--Moran constructed in \cite{Interpolationbypolynomialswithnonnegativecoefficients} an algorithm which tells whether there exists an interpolating CMP.
We can describe the relation between our algorithm and the algorithm in \cite{Interpolationbypolynomialswithnonnegativecoefficients} as follows:\\
In \cite{Interpolationbypolynomialswithnonnegativecoefficients}, they use the normalization $a_n=b_n=1$ and defined a clam $\mathcal{C}_{n-1}\subset \mathbb{R}^{n-1}$ as the convex hull of the set
\[
\{(a_{n-1}^i,\ldots, a_1^i)\}_{i=0}^{\infty}.
\]
Note that the order of coordinates adopted in \cite{Interpolationbypolynomialswithnonnegativecoefficients} is opposite to ours.
For $n\in \omega$, put
\begin{align*}
\scr{S}_{+,n}&:=\{s\in \scr{S}_{+}\mid \mathfrak{d}(s)=n\}\\
\scr{S}_{+,\leq n}&:=\{s\in \scr{S}_{+}\mid \mathfrak{d}(s)\leq n\}.
\end{align*}
For $s \in \scr{S}_{+,\leq n}$, let $\Delta_s^{(n-1)}\subset \mathbb{R}^{n-1}$ be an open simplex defined as the interior of the convex hull of $\{(a_{n-1}^i,\ldots, a_1^i)\mid i\in \supp(s)\}$ in its affine hull.
Then we have a simplicial decomposition
\[
\mathcal{C}_{n-1}=\coprod_{s\in \scr{S}_{+,\leq n}} \Delta_s^{(n-1)}.
\]
Moreover, $\{\Delta_s^{(n-1)} \mid s\in \scr{S}_{+,n-1}\}$ is precisely the set of surfaces of $\mathcal{C}_{n-1}$.

In the algorithm in \cite{Interpolationbypolynomialswithnonnegativecoefficients}, they choose an initial surface at first, and move to an adjacent surface one after another until they reach the surface intersecting the line through $(b_{n-1},...,b_2,b_1)$ parallel to the axis for the last coordinate.
(Note that surfaces are only on the top or only on the bottom according to the parity of $n$. See also \cite{Interpolationbypolynomialswithnonnegativecoefficients} \S 2 and its Definition 5 for their treatment of TOP and BOT.)
One can see that the surface they finally reach corresponds to the support of the minimal polynomial of $\{(a_i,b_i)\}_{i=2}^n$.
Therefore, the algorithm in \cite{Interpolationbypolynomialswithnonnegativecoefficients} for $\{(a_i,b_i)\}_{i=1}^n \cup \{(a_0,b_0)\}$ with $0<a_0<a_1$ and $0<b_0<b_1$ gives the minimal polynomial of $\{(a_i,b_i)\}_{i=1}^n$.

Our algorithm involves varying $s\in \scr{S}_{+}$ one after another by Algorithm 2, which corresponds to moving from one surface to another.
Thus the algorithms in \cite{Interpolationbypolynomialswithnonnegativecoefficients} and in this paper are similar in some way.
However, while the algorithm in \cite{Interpolationbypolynomialswithnonnegativecoefficients} determines the next surface by a greedy way, ours describes a more concrete form of the next surface.
Our algorithm makes it clear that the path is along the order in \S \ref{An order structure of sign sequences}.
In particular, it is clear from the construction that our method only involves shifts of sign sequences \textbf{to the right} (e.g. a shift like $++0\to 0++$), while it is not clear in \cite{Interpolationbypolynomialswithnonnegativecoefficients}. 
See also \cite{Interpolationbypolynomialswithnonnegativecoefficients} \S 3 for their analysis on ``adjacent hyperplanes."
\begin{rmk}
In the algorithm in this paper, the next surface we are moving to is not necessarily the adjacent one.
We may skip some surfaces.
\end{rmk}
\section{Open Problems}\label{Open Problems}
Theorem \ref{discreteupperbound} is optimal in the following sense: there is a sequence $S_{k}=\{(a_{k,i},b_{k,i})\}_{i=1}^{n_{k}}$ such that: 
\begin{itemize}
\item $n_{k}, A_{k}:=\max_{i=1}^{n_{k}}{\sz(a_{k,i})}$, and $B_{k}:= \max_{i=1}^{n_{k}}{\sz(b_{k,i})}$ are all $\leq \poly(k)$,
\item Each $S_{k}$ has a minimal polynomial $f_{k}$ with $\deg(f_{k}) = 2^{k^{\Omega(1)}}$.
\end{itemize}
Indeed, if we set 
\[n_{k}:=2,\ (a_{k,1},b_{k,1}):=(1,1) \ \mbox{and}\  (a_{k,2},b_{k,2}):=(1+2^{-k}, 2),\]
then, by Example \ref{baseexample}, each $S_{k}$ has a minimal polynomial $f_{k}$, and
$\deg f_{k} = \lceil \log2 / \log(1+2^{-k}) \rceil \geq 2^{k-1}$.

However, we do not know whether the dependence of the upper bound above on $n$ is crucial or not:
\begin{question}
Is there an upper bound of $\deg(f)$ of the form $2^{\poly(A)}\poly(B)$ (independent from $n$)?
Or, is there a sequence $S_{k}=\{(a_{k,i},b_{k,i})\}_{i=1}^{n_{k}}$ satisfying the following?: $n_{k},B_k = k^{O(1)}$, $A_{k}= O(\log(k))$, and each $S_{k}$ has a minimal polynomial $f_{k}$ with $\deg(f_{k}) = 2^{k^{\Omega(1)}}$.

\end{question}

Furthermore, the known algorithms require exponential time to decide whether a given $S$ has a minimal polynomial or not.
Moreover, because of the exponential degree upper bound (Theorem \ref{discreteupperbound}), the binary representation of the degree of a minimal polynomial is bounded by $\poly(n,A,B)$.
Therefore, the following are interesting:

\begin{question}
    Is there an algorithm to decide whether a given $S$ has a minimal polynomial or not in $\poly(n,A,B)$-time?
\end{question}

\begin{question}
     Given $S$ having a minimal polynomial, can we find the degree of the minimal polynomial in $\poly(n,A,B)$-time?
\end{question}


\section{Acknowledgement}
    The authors would like to thank Takahiro Ueoro and Masahiro Kurisaki for helpful discussions at the early stage of our research. 
    Further, we deeply appreciate anonymous referees' leading our attention to \cite{degreeupperboundforsparseinterpolation} and \cite{Interpolationbypolynomialswithnonnegativecoefficients}.
    We are also grateful to Kazuhiro Yokoyama for suggesting us to consider degree upper bounds for minimal polynomials, Pavel Hrube\v{s} for his sincere comments on the first version of our draft. 
    We also thank Yuki Ishihara, Yoshiyuki Sekiguchi, Toshiyasu Arai, Zden\v{e}k Strako\v{s}, and Mykyta Narusevych for their comments.

    This work is supported by JSPS KAKENHI Grant Number 22J22110 and 22J22505. The first author was supported by Grant-in-Aid for JSPS Fellows and FMSP, WINGS Program, the University of Tokyo. The second author was also supported by Grant-in-Aid for JSPS Fellows and FoPM, WINGS Program, the University of Tokyo.



\begin{thebibliography}{99}
   \bibitem{AlbouyFu} Albouy, A., \& Fu, Y. (2014).
Some remarks about Descartes' rule of signs.
\textit{Elem. Math}, \textit{69(4)}, 186--194.

\bibitem{DescartesRuleofSignoptimalityfirst}
Anderson, B., Jackson, J., \& Sitharam, M. (1998). Descartes’ rule of signs revisited. \textit{The American Mathematical Monthly}, \textit{105(5)}, 447--451.
  
  \bibitem{firstminimalpolynomial} Berg, L. (1985). Zur Identifikation mehrfach monotoner Funktionen. \textit{ZAMM}, \textit{65}, 497-507.
  \bibitem{minimalpolynomials}
  Berg, L. (1987). Existence, characterization and application of interpolating minimal polynomials.
\textit{Optimization}, \textit{18}, no. 1, 55--63.

  \bibitem{RAG}
  Bochnak, J., Coste, M., \& Roy, M.F. \textit{Real algebraic geometry}. (1998). Springer-Verlag, Berlin, Results in Mathematics and Related Areas (3), 36. 

   \bibitem{degreeupperboundforsparseinterpolation}
 Borodin, A., \& Tiwari, P. (1991).
On the decidability of sparse univariate polynomial interpolation.
\textit{Comput. Complexity}, \textit{1(1)}, 67--90.

 \bibitem{Interpolationbypolynomialswithnonnegativecoefficients}
 Bottema, M. J., Halpern, M. E., \& Moran, W. (1996). Interpolation by polynomials with nonnegative coefficients. In \textit{Proceedings of the IEEE Conference on Decision and Control}, \textit{4}, 2884--2889.

  \bibitem{FeliuTelek}Feliu,E., \& Telek, M.L. (2022).
On generalizing Descartes' rule of signs to hypersurfaces,
\textit{Advances in Mathematics}, \textit{408(A)}.

\bibitem{DescartesRuleofSignsoptimalcoolconstruction}
  Grabiner, D. J. (1999). Descartes’ rule of signs: another construction. \textit{The American Mathematical Monthly}, \textit{106(9)}, 854--856.

 \bibitem{Thielcke}
 Thielcke, B. (1986). Zwei Algorithmen zur Berechnung interpolierender Minimalpolynome. \textit{Rostock Math. Kolloq}, \textit{80}, 56--62.
\end{thebibliography}

\appendix




\end{document}